\documentclass{amsart}

\usepackage{amsthm}
\usepackage{amssymb, times}
\usepackage{amsmath}
\usepackage{amscd}
\usepackage{color}

\newcommand{\lb}{\varLambda}

\newcommand{\bg}{\begin{equation}}
\newcommand{\ed}{\end{equation}}
\newcommand{\bga}{\begin{eqnarray}}
\newcommand{\eda}{\end{eqnarray}}

\newtheorem {Theorem}  {Theorem}

\numberwithin{Theorem}{section}

\newtheorem {Lemma}[Theorem]  {Lemma}

\theoremstyle{definition}

\theoremstyle{remark}

\numberwithin{equation}{section}

\renewcommand{\th}{\theta}

\newcommand{\R}{\mathbf{R}}

\renewcommand{\div}{\mbox{div}}
\newcommand{\Aa}{{\mathcal A}}

\def  \R   {{\mathbb R}}

\def  \T   {{\mathbb T}}

\def  \12  {{\frac{1}{2}}}

\begin{document}

\title[Determining wavenumber for subcritical SQG]{On the determining wavenumber for the nonautonomous subcritical SQG equation}

\author [Alexey Cheskidov]{Alexey Cheskidov}
\address{Department of Mathematics, Stat. and Comp.Sci.,  University of Illinois Chicago, Chicago, IL 60607,USA}
\email{acheskid@uic.edu} 
\author [Mimi Dai]{Mimi Dai}
\address{Department of Applied Mathematics, Stat. and Comp.Sci.,  University of Illinois Chicago, Chicago, IL 60607,USA}
\email{mdai@uic.edu} 

\thanks{The work of A. Cheskidov was partially supported by NSF grant DMS--1517583. The work of M. Dai was partially supported by NSF Grant DMS--1815069.}

\begin{abstract}
A time-dependent determining wavenumber was introduced in \cite{CD-sqg1} to estimate the number of determining modes for the surface quasi-geostrophic (SQG) equation.
In this paper we continue this investigation focusing on the subcritical case and study trajectories  inside an absorbing set bounded in $L^\infty$. Utilizing this bound we find
a time-independent determining wavenumber that improves the estimate obtained in \cite{CD-sqg1}. This classical approach is more direct, but it is contingent on the existence
of the $L^\infty$ absorbing set.

\bigskip


\hspace{0.02cm}CLASSIFICATION CODE: 35Q35, 37L30.
\end{abstract}

\maketitle

\section{Introduction}
\label{sec-intro}

In this paper we estimate the number of determining modes for the forced subcritical surface quasi-geostrophic (SQG) equation  (see \cite{CMT})
\begin{equation}\begin{split}\label{QG}
\frac{\partial\th}{\partial t}+u\cdot\nabla \th+\nu\Lambda^\alpha\th =f,\\
u=R^\perp\th,
\end{split}
\end{equation}
where $x \in \mathbb T^2=[0,L]^2$, $1<\alpha<2$, $\nu>0$, $\Lambda=\sqrt{-\Delta}$ is the Zygmund operator, and
\bg\notag
R^\perp\th=\Lambda^{-1}(-\partial_2\th,\partial_1\th).
\ed
The initial data $\th(0) \in L^2(\mathbb{T}^2)$ and the force $f \in L^\infty(0,\infty; L^p(\mathbb{T}^2))$ for some $p>2/\alpha$ are assumed to have zero average.

A time-dependent determining wavenumber $\lb(t)$ was introduced in \cite{CD-sqg1} in the case
where $\alpha \in (0,2)$ and the force could be potentially rough. The determining wavenumber was
defined based only on the structure of the equation and  without any requirements on the regularity 
of solutions. It was shown that if two solutions coincide below $\lb(t)$,  the difference between them  decay exponentially,
even when they are far away from the attractor. Moreover, $\lb(t)$ was shown to be
uniformly bounded for all the solutions on the global attractor when $\alpha \in [1,2)$ and $f \in L^p$,
$p>2/\alpha$, in which case the attractor is bounded in $L^\infty$.  In this paper we investigate
this situation further and present a different, more direct approach in the subcritical case
$\alpha \in (1,2)$. Here we consider solutions that already entered an $L^\infty$ absorbing set and
take advantage of the $L^\infty$ bound (which is proportional to the $L^p$-norm of the force) to
define a time-independent  determining wavenumber $\lb$ and improve the final estimate for the number of
determining modes that we had in \cite{CD-sqg1}.
The drawback of this method is that it is less general and works only for regular solutions in the $L^\infty$
absorbing set. For a more complete background on the topic of finite dimensionality of flows, we refer the readers to \cite{CD-sqg1, CDK, CFMT, CFT, FJKT, FJKTgeneral,  FMRT, FMTT, FP,FT, FTiti} and references therein.

For the critical SQG equation ($\alpha=1$), due to the balance of the nonlinear term and the dissipative term, the global regularity problem was  challenging. However it was solved by different authors using different sophiscated methods in \cite{CaV, CCV, CVicol, KN09, KNV}. For the subcritical SQG equation with $1<\alpha<2$, the dissipative term dominates. 
In this case the global regularity was obtained in \cite{Re}.

In this paper, we will consider forces $f \in L^\infty(I; L^p(\mathbb{T}^2))$, $p>2/\alpha$, where $I=(0,\infty)$  or $(-\infty,\infty)$,  such that
\[
\sup_{t\in I} \|f(t)\|_p \leq F,
\]
for some fixed $F >0$. Then $\{\theta \in L^2: \|\theta\|_\infty \leq R_\infty\}$ is an absorbing set in $L^2$ (see Section~\ref{absorbing_sets}), where
\begin{equation} \label{eq:boundOnR}
R_\infty \sim  \lambda_0^{\frac{2}{p}-\alpha}\frac{F}{\nu}.
\end{equation}
Here $\lambda_0 = 1/L$. We prove the following.

\begin{Theorem}\label{thm}
Let $\alpha \in (1,2)$, $l>\frac{\alpha}{\alpha-1}$, and $Q \in \mathbb{N}$ be such that
\[
\lb := \lambda_0 2^Q \geq \left(\frac{Cl^2R_\infty}\nu\right)^{\frac1{\alpha-1}},
\]
where $C$ is some absolute constant.
 Let $\th_1(t)$ be a solution of \eqref{QG} with $f=f_1$ and $\th_2(t)$ be a solution to \eqref{QG}
with $f=f_2$. If
\begin{equation*} \label{eq:dm-condition}
\|\th_1(t)_{\leq Q}-\th_2(t)_{\leq Q}\|_{B^0_{l,l}} \to 0, \qquad \text{and} \qquad \|f_1-f_2\|_{B^{-\alpha(1-\frac{1}{l})}_{l,l}} \to 0, \qquad \text{as} \qquad  t \to \infty,
\end{equation*}
then
\begin{equation*} \label{eq:boundonldif-expon}
\|\th_1(t) - \th_2(t)\|_{B^0_{l,l}}^l \to 0 \qquad \text{as} \qquad t \to \infty.
\end{equation*}

Moreover, if $\th_1(t)$ and $\th_2(t)$ are two complete (ancient) solutions of \eqref{QG} with $f \in L^\infty((-\infty, \infty);L^p)$, $p>2/\alpha$, such that
$\th_1, \th_2 \in L^\infty((-\infty,\infty);L^2)$ and
\begin{equation} \label{eq:dm-condition}
\th_1(t)_{\leq Q}=\th_2(t)_{\leq Q}, \qquad \forall t<0,
\end{equation}
then
\[
\th_1(t) = \th_2(t), \qquad \forall t \in \mathbb{R}.
\]   
\end{Theorem}

The second part of the theorem concerns solutions on the pullback attractor 
\[
\Aa(t) = \{\th(t): \th(\cdot) \text{ is a complete bounded solution, i.e., } \th \in L^\infty((-\infty,\infty);L^2)\},
\]
that describes the long time behavior of solutions as the initial time goes to minus infinity. The fact that $\Aa(t)$ is indeed a pullback attractor follows, for example, from the general framework \cite{CK}. 

In the particular case of a time independent force $f \in L^p$, all the time slices of $\Aa(t)$ coincide, and
\[
\Aa = \Aa(t), \qquad \forall t \in \mathbb{R},
\]
is the global attractor.
Again, in the subcritical case $\alpha >1$,  it is easy to show that $\Aa$ is a global attractor by virtue of classical methods, or applying the evolutionary system
framework \cite{C} that requires the existence of an absorbing ball, energy inequality, and continuity of trajectories. This method does not require proving 
the existence of a compact absorbing set, and was used in \cite{CD} to show that $\Aa$ is the global attractor in the critical case $\alpha=1$ (see also \cite{CCV}
for the existence of the global attractor in $H^1$). In addition, in the autonomous case $f \in L^p$, the number of determining modes was estimated in \cite{CD-sqg1} using a much more general method applicable to subcritical, critical, and supercritical regimes. Theorem~\ref{thm} provides an improvement to the upper bound on $\lb$ in the subcritical case $\alpha >1$.

\bigskip

\section{Preliminaries}
\label{sec:pre}

\subsection{Notations}
We denote by $A\lesssim B$ an estimate of the form $A\leq C B$ with
some absolute constant $C$, and by $A\sim B$ an estimate of the form $C_1
B\leq A\leq C_2 B$ with some absolute constants $C_1$, $C_2$. 
We write $\|\cdot\|_p=\|\cdot\|_{L^p}$, and $(\cdot, \cdot)$ stands for the $L^2$-inner product.

\subsection{Littlewood-Paley decomposition}
\label{sec:LPD}
We recall briefly the Littlewood-Paley decomposition theory, which is one of the main techniques used in the paper. For a more detailed description on  this theory we refer readers to the books \cite{BCD, Gr}. 

Denote $\lambda_q=\frac{2^q}{L}$ for integers $q$. A nonnegative radial function $\chi\in C_0^\infty(\R^n)$ is chosen such that 
\begin{equation}\label{eq-xi}
\chi(\xi)=
\begin{cases}
1, \ \ \mbox { for } |\xi|\leq\frac{3}{4}\\
0, \ \ \mbox { for } |\xi|\geq 1.
\end{cases}
\end{equation}
Let 
\bg\notag
\varphi(\xi)=\chi(\xi/2)-\chi(\xi)
\ed
and
\begin{equation}\notag
\varphi_q(\xi)=
\begin{cases}
\varphi(2^{-q}\xi)  \ \ \ \mbox { for } q\geq 0,\\
\chi(\xi) \ \ \ \mbox { for } q=-1.
\end{cases}
\end{equation}
For a tempered distribution vector field $u$,  its  Littlewood-Paley projection $u_q$ is defined as follows.
\begin{equation}\notag
\begin{split}
& h_q:=\sum_{k\in \mathbb Z^n}\varphi_q(k)e^{i\frac{2\pi k\cdot x}{L}}\\
&u_q:=\Delta_qu=\sum_{k\in \mathbb Z^n}\hat u_k\varphi_q(k)e^{i\frac{2\pi k\cdot x}{L}}=\frac{1}{L^2}\int_{\T^2} h_q(y)u(x-y)dy,  \qquad q\geq -1,
\end{split}
\end{equation}
where $\hat{u}_k$ is the $k$th Fourier coefficient of $u$. 
Then we have
\bg\notag
u=\sum_{q=-1}^\infty u_q
\ed
in the distributional sense. We also denote
\bg\notag
u_{\leq Q}=\sum_{q=-1}^Qu_q, \qquad u_{(Q,R]}=\sum_{p=Q+1}^Ru_p, \qquad \tilde u_q=\sum_{|p-q|\leq 1}u_p. 
\ed
The Besov $B^s_{l,l}$-norm is defined as
\[
\|u\|_{B^s_{l,l}} = \left( \sum_{q=-1}^\infty \lambda_q^s \|u_q\|_l^l \right)^\frac{l}{l}.
\]
The following inequalities will be frequently used:
\begin{Lemma}\label{le:bern}(Bernstein's inequality) 
Let $n$ be the space dimension and $r\geq s\geq 1$. Then for all tempered distributions $u$, 
\bg\notag
\|u_q\|_{r}\leq \lambda_q^{n(\frac{1}{s}-\frac{1}{r})}\|u_q\|_{s}.
\ed
\end{Lemma}

\begin{Lemma}\label{lp}
Assume $2<l<\infty$ and $0\leq \alpha\leq 2$. Then
\begin{equation}\notag
l\int u_q\Lambda^\alpha u_q|u_q|^{l-2} \, dx\gtrsim  \lambda_q^\alpha\|u_q\|_l^l.
\end{equation}
\end{Lemma}
For a proof of Lemma \ref{lp}, see \cite{CMZ, CC}.

\subsection{Bony's paraproduct and commutator}
\label{sec-para}

Bony's paraproduct formula will be used to decompose the nonlinear terms. We will use the same version as in \cite{CD}:
\begin{equation}\notag
\begin{split}
\Delta_q(u\cdot\nabla v)=&\sum_{|q-p|\leq 2}\Delta_q(u_{\leq{p-2}}\cdot\nabla v_p)+
\sum_{|q-p|\leq 2}\Delta_q(u_{p}\cdot\nabla v_{\leq{p-2}})\\
&+\sum_{p\geq q-2} \Delta_q(\tilde u_p \cdot\nabla v_p).
\end{split}
\end{equation}
Some terms in this decomposition will be estimated using commutators. Let
\begin{equation} \label{eq:CommutatodDef}
[\Delta_q, u_{\leq{p-2}}\cdot\nabla]v_p:=\Delta_q(u_{\leq{p-2}}\cdot\nabla v_p)-u_{\leq{p-2}}\cdot\nabla \Delta_qv_p.
\end{equation}
By definition of $\Delta_q$ and Young's inequality,
\begin{equation}\label{commu}
\|[\Delta_q,u_{\leq{p-2}}\cdot\nabla] v_p\|_{r}
\lesssim \|\nabla u_{\leq p-2}\|_\infty\|v_p\|_{r},
\end{equation}
for any $r>1$ (see \cite{CD-sqg1} for details).

\bigskip

\section{Absorbing sets} \label{absorbing_sets}

First, we recall the $L^\infty$ estimates from \cite{CD-sqg1}. 
\begin{Lemma}
\label{Linfty}
Let $\alpha \in (0, 2)$ and $\th$ be a solution of \eqref{QG} on $[0,\infty)$ with $\th(0) \in L^2$ and
\[
\sup_{t>0} \|f(t)\|_p \leq F,
\]
for some $F\geq 0$ and $p\in(2/\alpha,\infty]$. Then,
for every $t>0$,
\begin{equation}\label{infty-norm}
\|\theta(t)\|_{L^\infty}\lesssim  \frac{\|\theta(0)\|_{2}}{(\nu t)^{\frac1\alpha}}+
\lambda_0^{\frac2p -\alpha} \frac{F}{\nu}\left(1+\lambda_0^{\frac{\alpha}{2}-1}(\nu t)^{\frac{1}{2}-\frac{1}{\alpha}}\right).
\end{equation}
\end{Lemma}
\begin{proof}
Identical to the proof of Lemma 4.2 in \cite{CD-sqg1} provided $\|f\|_p$ is replaced with $F$. 
\end{proof}

Due to the energy equality
\[
\|\th(t)\|_2^2 = \|\th(t_0)\|_2^2 + \int_{t_0}^t \left(-\nu \| \Lambda^{\frac{\alpha}{2}} \th(\tau) \|_2 +(f(\tau),\th(\tau))\right) \, d\tau, \qquad 0 \leq t_0 \leq t,
\]
and the fact that
\[
\|\Lambda^{-\frac{\alpha}{2}}f\|_2 \lesssim \lambda_0^{\frac{2}{p}-1-\frac{\alpha}{2}} \|f\|_p \leq \lambda_0^{\frac{2}{p}-1-\frac{\alpha}{2}} F,
\]
we have
\[
\|\th(t)\|_2^2 \lesssim \|\th(0)\|_2^2e^{-\nu(2\pi\lambda_0)^\alpha t}+\frac{\lambda_0^{\frac{4}{p}-2-2\alpha}F^2}{\nu^2}\left(1-e^{-\nu(2\pi\lambda_0)^\alpha t }\right), \qquad t>0.
\]
which implies the existence of an absorbing ball in $L^2$.
Indeed, for any bounded set $U \subset L^2$ there exists  time $t_{L^2}$, such that
\[
\th(t) \in B_{L^2}, \qquad \forall t \geq t_{L^2},
\]
for any solution $\th(t)$ with $\th(0) \in U$. Here
\[
B_{L^2} = \left\{ \th \in L^2: \|\th\|_2 \leq R_{2} \right\}, \qquad R_{2} \sim  \lambda_0^{\frac{2}{p}-1-\alpha} \frac{F}{\nu}.
\]
No consider the following ball in $L^\infty$:
\[
B_{L^\infty}=\left\{\th \in B_{L^2}: \|\th\|_\infty \leq  R_\infty \right\}, \qquad
R_\infty \sim \lambda_0^{\frac{2}{p}-\alpha}\frac{F}{\nu}.
\]
Lemma~\ref{Linfty} implies that $B_{L^\infty}$ is an absorbing set as well, i.e., for any bounded set $U \subset L^2$ there exists  time $t_{L^\infty}$, such that
\[
\th(t) \in B_{L^\infty}, \qquad \forall t\geq t_{L^\infty},
\]
for any solution $\th(t)$ with $\th(0) \in U$.

\qed

\section{Proof of the main result}
\label{sec:pf}

First we recall a generalization of Gr\"onwall's lemma from \cite{FMTT}.

\begin{Lemma} \label{L:GenGronwall}
Let $\alpha(t)$ be a locally integrable real valued function on $(0,\infty)$, satisfying for some $0 < T < \infty$  the following conditions:
\[
\liminf_{t \to \infty} \int_t^{T+t} \phi(\tau) \, d\tau >0, \qquad \limsup_{t \to \infty} \int_t^{T+t} \phi^-(\tau) \, d\tau < \infty,
\]
where $\phi^-= \max\{ -\phi, 0\}$. Let $\psi(t)$ be a measurable real valued function on $(0,\infty)$ such that
\[
\psi(t) \to 0, \qquad \text{as} \qquad t \to \infty.
\]
Suppose $\xi(t)$ is an absolutely continuous non-negative function on $(0, \infty)$ such that
\[
\frac{d}{dt} \xi + \phi \xi \leq \psi, \qquad \text{a.e. on} \ (0,\infty).
\]
Then
\[
\xi(t) \to 0 \qquad \text{as} \qquad t \to \infty.
\]
\end{Lemma}

Now we are ready to prove the main result.
\bigskip

\noindent
{\em Proof of Theorem~\ref{thm}.}
Consider two solution $\theta_1$, $\theta_2$ of \eqref{QG} with forces $f_1$ and $f_2$. Let $t_0$ be a time after which the solutions stay in the absorbing set $B_{L^\infty}$:
\[
\|\theta_1(t)\|_\infty \leq R_\infty, \qquad \|\theta_2(t)\|_\infty \leq R_\infty, \qquad t \geq t_0.
\]
In what follows we assume that $t\geq t_0$. Denote $u_1=R^\perp\th_1$ and $u_2=R^\perp\th_2$.
Let $f=f_1-f_2$ and $w=\th_1-\th_2$, which satisfies the equation
\begin{equation} \label{eq-w}
w_t+u_1\cdot\nabla w+\nu\Lambda^\alpha w+R^\perp w\cdot\nabla \th_2=f.
\end{equation}

Projecting equation (\ref{eq-w}) onto the $q$-th shell, multiplying by $lw_q|w_q|^{l-2}$, integrating, adding up for all $q\geq -1$, applying Lemma \ref{lp}, H\"older and Young inequalities, yield
\begin{equation}\label{w2}
\begin{split}
\frac{d}{dt} \|w(t)\|_{B^0_{l,l}}^l  + C \nu \|\Lambda^{\alpha/l} w\|_{B^0_{l,l}}^l -\left(\frac{2}{C\nu} \right)^{l-1}l^{l-2}\| \Lambda^{-\alpha\left(1-\frac{1}{l}\right)}& f\|^l_{B^0_{l,l}} \leq \\
 -l\sum_{q\geq -1}\int_{\T^3}\Delta_q(R^\perp w\cdot\nabla \th_2) & w_q|w_q|^{l-2}\, dx\\
- l\sum_{q\geq -1}\int_{\T^3}\Delta_q(u_1 \cdot\nabla w) & w_q|w_q|^{l-2}\, dx \\
=& I +  J,
\end{split}
\end{equation}
for some absolute constant $C$.
Using Bony's paraproduct formula, $I$ is decomposed as
\begin{equation}\notag
\begin{split}
I=&
-l\sum_{q\geq -1}\sum_{|q-p|\leq 2}\int_{\T^3}\Delta_q(R^\perp w_{\leq{p-2}}\cdot\nabla(\th_2)_p) w_q|w_q|^{l-2}\, dx\\
&-l\sum_{q\geq -1}\sum_{|q-p|\leq 2}\int_{\T^3}\Delta_q(R^\perp w_{p}\cdot\nabla (\th_2)_{\leq{p-2}})w_q|w_q|^{l-2}\, dx\\
&-l\sum_{q\geq -1}\sum_{p\geq q-2}\int_{\T^3}\Delta_q(R^\perp\tilde w_p\cdot\nabla (\th_2)_p)w_q|w_q|^{l-2}\, dx\\
 =& I_{1}+I_{2}+I_{3}.
\end{split}
\end{equation}
Recall that $\lb=2^{Q}/L$. To estimate $I_1$ we use H\"older's inequality and split it as follows:
\begin{equation}\notag
\begin{split}
|I_{1}|&\leq l\sum_{q\geq -1}\sum_{|q-p|\leq 2}\int_{\T^3}\left|\Delta_q(R^\perp w_{\leq{p-2}}\cdot\nabla(\th_2)_p) w_q\right||w_q|^{l-2}\, dx\\
&\lesssim l\sum_{q> Q}\| w_q\|_l^{l-1}\sum_{|q-p|\leq 2}\lambda_p\|(\th_2)_p\|_\infty\sum_{Q<p'\leq p-2}\|R^\perp w_{p'}\|_l\\
&+l\sum_{q> Q}\| w_q\|_l^{l-1}\sum_{|q-p|\leq 2}\lambda_p\|(\th_2)_p\|_\infty\|R^\perp w_{\leq Q}\|_l\\
&+l\sum_{q\leq Q}\| w_q\|_l^{l-1}\sum_{|q-p|\leq 2}\lambda_p\|(\th_2)_p\|_\infty\|R^\perp w_{\leq p-2}\|_l\\
&\equiv I_{11}+I_{12}+I_{13}.
\end{split}
\end{equation}
 Then using Young's inequality, Jensen's inequality and the fact that $\|R^\perp w_q\|_l\lesssim \|w_q\|_l$, we obtain
\begin{equation}\notag
\begin{split}
|I_{11}|
&\lesssim R_\infty l\sum_{p> Q-2}\lambda_p\sum_{|q-p|\leq 2}\| w_q\|_l^{l-1}\sum_{Q<p'\leq p-2}\|R^\perp w_{p'}\|_l\\
&\lesssim R_\infty l\sum_{p> Q}\lambda_p\| w_p\|_l^{l-1}\sum_{Q<p'\leq p-2}\|R^\perp w_{p'}\|_l\\
&\lesssim \lb^{1-\alpha+\frac\alpha l}R_\infty l \sum_{p> Q}\lambda_p^{\frac{\alpha(l-1)}{l}}\| w_p\|_l^{l-1}\sum_{Q<p'\leq p-2}\|R^\perp w_{p'}\|_l\\
&\lesssim \lb^{1-\alpha} R_\infty l\sum_{p> Q}\lambda_p^{\frac{\alpha(l-1)}{l}}\| w_p\|_l^{l-1}\sum_{Q<p'\leq p-2}\lambda_{p'}^{\frac \alpha l}\|R^\perp w_{p'}\|_l\lambda_{p'-Q}^{-\frac \alpha l}\\
&\lesssim \lb^{1-\alpha} R_\infty l\sum_{q>Q}\lambda_q^\alpha\|w_q\|_l^l,
\end{split}
\end{equation}
where we needed $1-\alpha+\frac\alpha l<0$, i.e., $l>\alpha/(\alpha-1)$.  Now we take small enough $\epsilon>0$, such that $1-\alpha+\frac\alpha l+\epsilon<0$,
and use H\"older's inequality, Young's inequality, and Jensen's inequality to infer
\[
\begin{split}
I_{12}&\lesssim R_\infty l\sum_{q> Q}\| w_q\|_l^{l-1}\sum_{|q-p|\leq 2}\lambda_p\|R^\perp w_{\leq Q}\|_l\\
&\lesssim R_\infty l\sum_{q> Q-2}\lambda_q\| w_q\|_l^{l-1}\|R^\perp w_{\leq Q}\|_l\\
&= R_\infty l\sum_{q> Q-2}\lambda_q^{1-\alpha+\frac\alpha l+\epsilon}\lambda_q^{-\epsilon}\lambda_q^{\frac{\alpha(l-1)}{l}}\| w_q\|_l^{l-1}\|R^\perp w_{\leq Q}\|_l\\
&\lesssim \lb^{1-\alpha+\frac\alpha l+\epsilon}R_\infty l\sum_{q> Q-2}\lambda_q^{-\epsilon}\lambda_q^{\frac{\alpha(l-1)}{l}}\| w_q\|_l^{l-1}\|R^\perp w_{\leq Q}\|_l\\
&\lesssim \lb^{1-\alpha}R_\infty l\left(\sum_{q> Q-2}\lambda_q^{-\epsilon}\lambda_q^{\frac{\alpha(l-1)}{l}}\| w_q\|_l^{l-1}\right)^{\frac{l}{l-1}}+\lb^{1+\epsilon l} R_\infty l\|R^\perp w_{\leq Q}\|_l^l\\
&\lesssim \lb^{1-\alpha-\epsilon}R_\infty l\sum_{q> Q-2}\lambda_q^\alpha\| w_q\|_l^{l}+\lb^{1+\epsilon l} R_\infty l\|w_{\leq Q}\|_l^l;
\end{split}
\]
and similarly,
\[
\begin{split}
I_{13}&\lesssim R_\infty l\sum_{q\leq Q}\| w_q\|_l^{l-1}\sum_{|q-p|\leq 2}\lambda_p\|R^\perp w_{\leq p-2}\|_l\\
&\lesssim R_\infty l\sum_{q\leq Q}\lambda_q\| w_q\|_l^{l-1}\|R^\perp w_{\leq Q}\|_l\\
&= R_\infty l\sum_{q\leq Q}\lambda_q^{1-\alpha+\frac\alpha l+\epsilon}\lambda_q^{-\epsilon}\lambda_q^{\frac{\alpha(l-1)}{l}}\| w_q\|_l^{l-1}\|R^\perp w_{\leq Q}\|_l\\
&\lesssim \lb^{1-\alpha}R_\infty l\sum_{q\leq Q}\lambda_q^{-\epsilon}\lambda_q^{\frac{\alpha(l-1)}{l}}\| w_q\|_l^{l-1}\lb^{\alpha-1}\|R^\perp w_{\leq Q}\|_l\\
&\lesssim \lb^{1-\alpha}R_\infty l\left(\sum_{q\leq Q}\lambda_q^{-\epsilon}\lambda_q^{\frac{\alpha(l-1)}{l}}\| w_q\|_l^{l-1}\right)^{\frac l{l-1}}+\lb^{(l-1)(\alpha-1)} R_\infty l\|R^\perp w_{\leq Q}\|_l^l\\
&\lesssim \lb^{1-\alpha}R_\infty l\sum_{q\leq Q}\lambda_q^\alpha\| w_q\|_l^{l}+\lb^{(l-1)(\alpha-1)} R_\infty l\|w_{\leq Q}\|_l^l.
\end{split}
\]
For $I_2$, splitting the summation and using H\"older's inequality, we obtain 
\[
\begin{split}
|I_{2}|&\lesssim l\sum_{q\geq -1}\sum_{|q-p|\leq 2}\int_{\R^3}|\Delta_q(R^\perp w_{p}\cdot\nabla (\th_2)_{\leq{p-2}})w_q||w_q|^{l-2}\, dx\\
&\lesssim l\sum_{q> Q-2}\sum_{|p-q|\leq2}\sum_{p'\leq p-2}\lambda_{p'}\|(\th_2)_{p'}\|_\infty\|R^\perp w_p\|_l\|w_q\|_l^{l-1}\\
&+l\sum_{q\leq Q-2}\sum_{|p-q|\leq2}\sum_{p'\leq p-2}\lambda_{p'}\|(\th_2)_{p'}\|_\infty\|R^\perp w_p\|_l\|w_q\|_l^{l-1}\\
&\equiv I_{21}+I_{22}.
\end{split}
\]
The first term is estimated as 
\[
\begin{split}
I_{21}
&\lesssim  R_\infty l\sum_{p>Q-4} \|w_p\|_l^l\sum_{p'\leq p-2}\lambda_{p'}\\
&\lesssim R_\infty l\sum_{p>Q-4}\lambda_p^\alpha \| w_p\|_l^l\sum_{p'\leq p-2}\lambda_{p'-p}\lambda_p^{1-\alpha}\\
&\lesssim \lb^{1-\alpha}R_\infty l\sum_{p>Q-4}\lambda_p^\alpha \| w_p\|_l^l.
\end{split}
\]
For the second term we have 
\[
\begin{split}
I_{22}&\lesssim R_\infty l\sum_{q\leq Q-2}\sum_{|p-q|\leq2}\sum_{p'\leq p-2}\lambda_{p'}\|R^\perp w_p\|_l\|w_q\|_l^{l-1}\\
&\lesssim R_\infty l\sum_{q\leq Q}\|w_q\|_l^{l}\sum_{p'\leq Q}\lambda_{p'}\\
&\lesssim \lb R_\infty l\sum_{q\leq Q}\|w_q\|_l^{l}.
\end{split}
\]
To estimate $I_{3}$, we first integrate by parts and then
use H\"older's inequality obtaining
\[
\begin{split}
|I_{3}|&\leq l\sum_{q\geq -1}\sum_{p\geq q-2}\int_{\mathbb R^3}|\Delta_q(R^\perp\tilde w_p (\th_2)_p) \nabla (w_q|w_q|^{l-2})| \, dx\\
&\lesssim l^2\sum_{q\geq -1}\sum_{p\geq q-2}\int_{\mathbb R^3}|\Delta_q(R^\perp\tilde w_p (\th_2)_p) \nabla w_q||w_q|^{l-2} \, dx\\
&\lesssim l^2\sum_{q>Q}\lambda_q\|w_q\|_{l}^{l-1}\sum_{p\geq q-2}\|R^\perp\tilde w_p\|_l\|(\th_2)_p\|_\infty\\
&+l^2\sum_{q\leq Q}\lambda_q\|w_q\|_{l}^{l-1}\sum_{p\geq q-2}\|R^\perp\tilde w_p\|_l\|(\th_2)_p\|_\infty\\
&\equiv I_{31}+I_{32}.
\end{split}
\]
For the first term we use Jensen's inequality:
\[
\begin{split}
I_{31}&\lesssim R_\infty l^2\sum_{p> Q-3}\|R^\perp w_p\|_l\sum_{Q<q\leq p+2}\lambda_q\|w_q\|_l^{l-1}\\
&\lesssim R_\infty l^2\sum_{p> Q-3}\lambda_p^{\frac\alpha l}\|w_p\|_l\sum_{Q<q\leq p+2}\lambda_q^{\frac{\alpha(l-1)}{l}}\|w_q\|_l^{l-1}\lambda_q^{1-\alpha}\lambda_{q-p}^{\frac\alpha l}\\
&\lesssim \lb^{1-\alpha}R_\infty l^2\sum_{q> Q-3}\lambda_q^\alpha\|w_q\|_l^l.
\end{split}
\]
For the second term, H\"older's inequality, Young's inequality, and Jensen's inequality yield
\[
\begin{split}
I_{32}&\lesssim R_\infty l^2\sum_{q\leq Q}\lambda_q\|w_q\|_l^{l-1}\sum_{p\geq q-2}\|R^\perp \tilde w_p\|_l\\
&\lesssim R_\infty l^2\sum_{q\leq Q}\|w_q\|_l^{l-1}\sum_{p\geq q-2}\lambda_p^{\frac\alpha l}\|R^\perp \tilde w_p\|_l\lambda_{p-q}^{-\frac\alpha l}\lambda_q^{1-\frac\alpha l}\\
&\lesssim \lb^{1-\frac\alpha l} R_\infty l^2\sum_{q\leq Q}\|w_q\|_l^{l-1}\sum_{p\geq q-2}\lambda_p^{\frac\alpha l}\|R^\perp \tilde w_p\|_l\lambda_{p-q}^{-\frac\alpha l}\\
&\lesssim \lb R_\infty l^2\sum_{q\leq Q}\left(\|w_q\|_l^{l}+\lb^{-\alpha}\left(\sum_{p\geq q-2}\lambda_p^{\frac\alpha l}\|R^\perp \tilde w_p\|_l\lambda_{p-q}^{-\frac\alpha l}\right)^l\right)\\
&\lesssim \lb R_\infty l^2\sum_{q\leq Q}\|w_q\|_l^{l}+\lb^{1-\alpha}R_\infty l^2\sum_{q\leq Q}\left(\sum_{p\geq q-2}\lambda_p^{\frac\alpha l}\|R^\perp \tilde w_p\|_l\lambda_{p-q}^{-\frac\alpha l}\right)^l\\
&\lesssim \lb^{1-\alpha}R_\infty l^2\sum_{q \geq -1}\lambda_q^\alpha\|w_q\|_l^l+ \lb R_\infty l^2\sum_{q\leq Q}\|w_q\|_l^{l}.
\end{split}
\]
Therefore, for $l$ such that $1-\alpha+\frac\alpha l<0$ we have
\begin{equation}\label{est-i1}
\begin{split}
|I| &\lesssim \lb^{1-\alpha}R_\infty l^2\sum_{q\geq -1}\lambda_q^\alpha\|w_q\|_l^l+\left(\lb^{(l-1)(\alpha-1)}+\lb^{1+\epsilon l}\right)R_\infty l^2 \sum_{q\leq Q}\|w_q\|_l^{l}\\
&\lesssim \lb^{1-\alpha}R_\infty l^2\sum_{q\geq -1}\lambda_q^\alpha\|w_q\|_l^l+\lb^{(l-1)(\alpha-1)}R_\infty l^2 \sum_{q\leq Q}\|w_q\|_l^{l},
\end{split}
\end{equation}
where $\epsilon$ is chosen  small enough so that $1-\alpha+\frac\alpha l+\epsilon<0$ and hence
$(l-1)(\alpha-1)>1+\epsilon l$.

We now estimate $J$, where we first apply Bony's paraproduct formula:
\begin{equation}\notag
\begin{split}
J=
&-l\sum_{q\geq -1}\sum_{|q-p|\leq 2}\int_{\R^3}\Delta_q((u_1)_{\leq{p-2}}\cdot\nabla w_p) w_q|w_q|^{l-2}\, dx\\
&-l\sum_{q\geq -1}\sum_{|q-p|\leq 2}\int_{\R^3}\Delta_q((u_1)_{p}\cdot\nabla w_{\leq{p-2}})w_q|w_q|^{l-2}\, dx\\
&-l\sum_{q\geq -1}\sum_{p\geq q-2}\int_{\R^3}\Delta_q((u_1)_p\cdot\nabla\tilde w_p)w_q|w_q|^{l-2}\, dx\\
=&J_{1}+J_{2}+J_{3}.
\end{split}
\end{equation}
Observing that $\sum_{|p-q|\leq 2}\Delta_qw_p=w_q$, we then decompose $J_{1}$ using the commutator notation \eqref{eq:CommutatodDef}:
\begin{equation}\notag
\begin{split}
J_{1}=&-l\sum_{q\geq -1}\sum_{|q-p|\leq 2}\int_{\R^3}[\Delta_q, (u_1)_{\leq{p-2}}\cdot\nabla] w_pw_q|w_q|^{l-2}\, dx\\
&-l\sum_{q\geq -1}\int_{\R^3} (u_1)_{\leq q-2}\cdot\nabla w_q w_q|w_q|^{l-2}\, dx\\
&-l\sum_{q\geq -1}\sum_{|q-p|\leq 2}\int_{\R^3}( (u_1)_{\leq{p-2}}- (u_1)_{\leq q-2})\cdot\nabla\Delta_qw_p w_q|w_q|^{l-2}\, dx\\
=&J_{11}+J_{12}+J_{13}.
\end{split}
\end{equation}
The term $J_{12}$ vanishes because $\div\, (u_1)_{\leq q-2}=0$. 
To estimate $J_{11}$ we will use \eqref{commu},   
\begin{equation}\notag
\|[\Delta_q, (u_1)_{\leq{p-2}}\cdot\nabla] w_p\|_l\\
\lesssim \|\nabla  (u_1)_{\leq p-2}\|_\infty\|w_p\|_l.
\end{equation}
Then splitting the summation we get
\begin{equation}\notag
\begin{split}
|J_{11}|
&\leq l\sum_{q\geq -1}\sum_{|q-p|\leq 2}\|[\Delta_q, (u_1)_{\leq{p-2}}\cdot\nabla] w_p\|_l\|w_q\|_l^{l-1}\\
&\leq l\sum_{q>Q-2}\sum_{|q-p|\leq 2}\|\nabla (u_1)_{\leq p-2}\|_\infty\|w_p\|_l\|w_q\|_l^{l-1}\\
&+l\sum_{q\leq Q-2}\sum_{|q-p|\leq 2}\|\nabla (u_1)_{\leq p-2}\|_\infty\|w_p\|_l\|w_q\|_l^{l-1}\\
&\equiv J_{111}+J_{112}.
\end{split}
\end{equation}
Now note that $\| (u_1)_q\|_\infty\lesssim \|(\th_1)_q\|_\infty \leq R_\infty$. So
using H\"older's and Bernstein's inequalities, we obtain
\begin{equation}\notag
\begin{split}
J_{111}&\lesssim l\sum_{q> Q-2}\sum_{|q-p|\leq 2}\sum_{p'\leq q}\lambda_{p'}\|(u_1)_{p'}\|_\infty\|w_p\|_l\|w_q\|_l^{l-1}\\
&\lesssim R_\infty l\sum_{q> Q-2}\sum_{|q-p|\leq 2}\sum_{p'\leq q}\lambda_{p'}\|w_p\|_l\|w_q\|_l^{l-1}\\
&\lesssim R_\infty l\sum_{q> Q-4}\sum_{p'\leq q}\lambda_{p'}\|w_q\|_l^l\\
&\lesssim R_\infty l\sum_{q> Q-4}\lambda_q^\alpha\|w_q\|_l^l\sum_{p'\leq q}\lambda_{p'-q}\lambda_q^{1-\alpha}\\
&\lesssim \lb^{1-\alpha}R_\infty l\sum_{q> Q-4}\lambda_q^\alpha\|w_q\|_l^l.
\end{split}
\end{equation}
Similarly,
\begin{equation}\notag
\begin{split}
J_{112}&\lesssim l\sum_{q\leq Q-2}\sum_{|q-p|\leq 2}\sum_{p'\leq p-2}\lambda_{p'}\|(u_1)_{p'}\|_\infty\|w_p\|_l\|w_q\|_l^{l-1}\\
&\lesssim R_\infty l\sum_{q\leq Q-2}\sum_{|q-p|\leq 2}\sum_{p'\leq q}\lambda_{p'}\|w_p\|_l\|w_q\|_l^{l-1}\\
&\lesssim R_\infty l\sum_{q\leq Q}\|w_q\|_l^{l}\sum_{p'\leq Q-2}\lambda_{p'}\\
&\lesssim \lb R_\infty l\sum_{q\leq Q}\|w_q\|_l^{l}.
\end{split}
\end{equation}
To estimate $J_{13}$, we first use H\"older's inequality and split the summation as follows:
\begin{equation}\notag
\begin{split}
|J_{13}|&\leq l\sum_{q\geq -1}\sum_{|q-p|\leq 2}\int_{\R^3}\left|( (u_1)_{\leq{p-2}}- (u_1)_{\leq q-2})\cdot\nabla\Delta_qw_p \right||w_q|^{l-1}\, dx\\
&\lesssim l\sum_{q>Q-2}\sum_{|q-p|\leq 2}\sum_{q-3\leq p'\leq q}\int_{\R^3}|(u_1)_{p'}||\nabla\Delta_qw_p ||w_q|^{l-1}\, dx\\
&+l\sum_{q\leq Q-2}\sum_{|q-p|\leq 2}\sum_{q-3\leq p'\leq q}\int_{\R^3}|(u_1)_{p'}||\nabla\Delta_qw_p ||w_q|^{l-1}\, dx\\
&\equiv J_{131}+J_{132}.
\end{split}
\end{equation}
Now Jensen's inequality yields
\begin{equation}\notag
\begin{split}
J_{131}
&\lesssim l\sum_{q>Q-2}\|w_q\|_l^{l-1}\sum_{|q-p|\leq 2}\lambda_p\|w_p\|_l\sum_{q-3\leq p'\leq q}\|(u_1)_{p'}\|_\infty\\
&\lesssim  R_\infty l\sum_{q>Q-2}\|w_q\|_l^{l-1}\sum_{|q-p|\leq 2}\lambda_p\|w_p\|_l\\
&\lesssim  R_\infty l\sum_{q>Q-2}\lambda_q^{\frac{\alpha(l-1)}{l}}\|w_q\|_l^{l-1}\sum_{|q-p|\leq 2}\lambda_p^{\frac{\alpha}{l}}\|w_p\|_l\lambda_{p-q}^{1-\frac{\alpha}{l}}\lambda_q^{1-\alpha}\\
&\lesssim  \lb^{1-\alpha}R_\infty l\sum_{q>Q-4}\lambda_q^{\alpha}\|w_q\|_l^{l}.
\end{split}
\end{equation}
And similarly, for the second term,
\begin{equation}\notag
\begin{split}
J_{132}
&\lesssim l\sum_{q\leq Q-2}\|w_q\|_l^{l-1}\sum_{|q-p|\leq 2}\lambda_p\|w_p\|_l\sum_{q-3\leq p'\leq q}\|(u_1)_{p'}\|_\infty\\
&\lesssim  R_\infty l\sum_{q\leq Q-2}\|w_q\|_l^{l-1}\sum_{|q-p|\leq 2}\lambda_p\|w_p\|_l\\
&\lesssim  R_\infty l\sum_{q\leq Q}\lambda_q\|w_q\|_l^{l}\\
&\lesssim  \lb R_\infty l\sum_{q\leq Q}\|w_q\|_l^{l}.
\end{split}
\end{equation}
For $J_2$ we use H\"older's inequality obtaining
\begin{equation}\notag
\begin{split}
|J_{2}| &\leq l\sum_{q\geq -1}\sum_{|q-p|\leq 2}\int_{\R^3}\left|\Delta_q((u_1)_{p}\cdot\nabla w_{\leq{p-2}})\right||w_q|^{l-1}\, dx\\
&\lesssim l\sum_{q>Q}\|w_q\|_l^{l-1}\sum_{|q-p|\leq 2}\|(u_1)_p\|_\infty\sum_{p'\leq p-2}\lambda_{p'}\|w_{p'}\|_l\\
&+l\sum_{q\leq Q}\|w_q\|_l^{l-1}\sum_{|q-p|\leq 2}\|(u_1)_p\|_\infty\sum_{p'\leq p-2}\lambda_{p'}\|w_{p'}\|_l\\
&\equiv J_{21}+J_{22}.
\end{split}
\end{equation}
Recall that  $\| (u_1)_q\|_\infty\lesssim \|(\th_1)_q\|_\infty\leq R_\infty$. Hence we can use Jensen's inequality to deduce that
\begin{equation}\notag
\begin{split}
J_{21}
&\lesssim R_\infty l\sum_{q>Q}\|w_q\|_l^{l-1}\sum_{p'\leq q}\lambda_{p'}\|w_{p'}\|_l\\
&\lesssim R_\infty l \sum_{q>Q}\lambda_q^{\frac{\alpha(l-1)}{l}}\|w_q\|_l^{l-1}\sum_{p'\leq q}\lambda_{p'}^{\frac\alpha l}\|w_{p'}\|_l\lambda_{p'-q}^{1-\frac\alpha l}\lambda_q^{1-\alpha}\\
&\lesssim \lb^{1-\alpha}R_\infty l\sum_{q>Q}\lambda_q^\alpha\|w_q\|_l^l,
\end{split}
\end{equation}
where we needed $l>\alpha$.  While the second term is estimated as
\begin{equation}\notag
\begin{split}
J_{22}
&\lesssim R_\infty l\sum_{q\leq Q}\|w_q\|_l^{l-1}\sum_{p'\leq q}\lambda_{p'}\|w_{p'}\|_l\\
&\lesssim \lb R_\infty l \sum_{q\leq Q}\|w_q\|_l^{l-1}\sum_{p'\leq Q}\|w_{p'}\|_l\\
&\lesssim \lb QR_\infty l\sum_{q\leq Q}\|w_q\|_l^l.
\end{split}
\end{equation}
Since $\|(u_1)_q\|_{r}\lesssim \|(\th_1)_q\|_{r}$ for any $r\in(1,\infty]$, the term
$J_{3}$ enjoys the same estimate as $I_{3}$. Hence we conclude that
\begin{equation}\label{est-i2}
|J| \lesssim \lb^{1-\alpha}R_\infty l^2\sum_{q\geq -1}\lambda_q^\alpha\|w_q\|_l^l+\lb QR_\infty l^2\sum_{q\leq Q}\|w_q\|_l^l.
\end{equation}

Thanks to (\ref{w2})--(\ref{est-i2}), inequality \eqref{w2} yields
\begin{equation}\notag
\begin{split}
\frac{d}{dt} \|w(t)\|_{B^0_{l,l}}^l  &\leq  - C \nu\|\Lambda^{\alpha/l} w\|_{B^0_{l,l}}^l
+C_1\lb^{1-\alpha}R_\infty l^2 \sum_{q>Q}\lambda_q^\alpha\|w_q\|_l^l \\
&+ \left(\frac{2}{C\nu} \right)^{l-1}l^{l-2}\| \Lambda^{-\alpha\left(1-\frac{1}{l}\right)} f\|^l_{B^0_{l,l}} + C_2\lb^{(l-1)(\alpha-1)}R_\infty l^2 \sum_{q\leq Q}\|w_q\|_l^l,
\end{split}
\end{equation}
for some absolute constants $C$, $C_1$, and $C_2$. Thus we have
\[
\frac{d}{dt} \|w(t)\|_{B^0_{l,l}}^l + \phi \|w(t)\|_{B^0_{l,l}}^l \leq \psi(t),
\]
where
\[
\begin{split}
\phi=& {\textstyle \frac12}(2\pi \lambda_0)^\alpha C\nu,\\
\psi(t)=& \left(\frac{2}{C\nu} \right)^{l-1}l^{l-2}\| \Lambda^{-\alpha\left(1-\frac{1}{l}\right)} f\|^l_{B^0_{l,l}} + C_2\lb^{(l-1)(\alpha-1)}R_\infty l^2 \sum_{q\leq Q}\|w_q\|_l^l,
\end{split}
\]
provided
\[
\lb = \left(\frac{2C_1l^2}{C\nu} R_\infty\right)^\frac{1}{\alpha-1}.
\]
Note that
\[
\psi(t) \to 0 \qquad \text{as} \qquad t \to \infty,
\]
due to the assumption of the theorem. Since also $\alpha >0$,
the first part of the theorem follows from Lemma~\ref{L:GenGronwall}.

To prove the second part, where $\psi(t) \equiv 0$,
we note that
\[
\|w(t)\|_{B^0_{l,l}}^l \leq \|w(t_0)\|_{B^0_{l,l}}^l e^{-\alpha (t-t_0)}, \qquad t_0 \leq t \leq 0, 
\]
thanks to Gr\"onwall's inequality.
Since $\th_1$  and $\th_2$ are ancient solutions,
\[
\th_1(t), \th_2(t) \in B_{L^2}\cap B_{L^\infty}, \qquad \forall t \leq 0.
\]
Hence, we have
\[
\begin{split}
\|w(t)\|_{B^0_{l,l}} &\lesssim \|w(t)\|_\infty^{1-\frac{2}{l}} \|w(t)\|_2^{\frac{2}{l}}\\
&\lesssim R_{\infty}^{1-\frac{2}{l}} R_{2}^{\frac{2}{l}},
\end{split}
\]
for all $t$. Taking the limit as $t_0 \to -\infty$ gives $w(t) =0$ for all $ t \leq 0$, and hence $w \equiv 0$.

\qed

\section*{Acknowledgement}
The authors would wish to express their gratitude to the anonymous referee for the careful review and valuable suggestions which helped to improve the manuscript a lot.

\end{document}